\documentclass[11pt,twoside]{article}
\usepackage{amsfonts,graphicx,lscape}
\usepackage{graphics}
\usepackage{epsf}
\usepackage{epsfig}

\usepackage{amssymb}
\usepackage{eucal}
\usepackage{amsmath}
\usepackage{amsthm}
\usepackage{amscd}
\usepackage{latexsym}
\usepackage{booktabs}
\usepackage{epstopdf}
\usepackage{float}
\usepackage[all]{xy}
  
\newtheorem{thm}{Theorem}[section]

\newtheorem{proposition}[thm]{Proposition}
\newtheorem{remark}[thm]{Remark}
\title{\bf ON A DEFORMED VERSION\\ OF THE $T$ SYSTEM\footnote{Preprint-Proceedings of the $XV^{th}$ International Conference on Mathematics and its Applications, Politehnica University of Timi\c{s}oara, 
November 1-3, 2018, pag. 90--97}}

\author{Cristian L\u azureanu and Cristiana C\u aplescu\\
Politehnica University of Timi\c{s}oara}

\date{}

\begin{document}

\thispagestyle{plain}

\maketitle

\begin{abstract}
We use the integrable deformations method for a three-dimensional
system of differential equations to obtain deformations of the $T$ system. 
We analyze a deformation given by particular deformation functions. We point out  that the obtained system preserves the chaotic behavior for some values of the deformation parameter.
\end{abstract}
\bigskip

\section{Introduction}
\smallskip
The systems of differentiable equations describe the dynamics of many phenomena that appear in engineering, economy, physics, biology, chemistry, and other domains. Some of such systems are perturbations of related, simpler systems, particularly integrable systems. A three-dimensional system of ODEs is integrable if there are two functionally independent smooth functions that are constants of motion of it. These systems are Hamilton-Poisson systems. Recently, integrable deformations of some Hamilton-Poisson systems were constructed. 

In \cite{BBM}, systems endowed with Lie-Poisson symmetries are considered. Considering Poisson-Lie groups as deformations of Lie-Poisson (co)algebras, a general method of construction of integrable deformations was proposed. In \cite{EKV17} it was shown that the Casimirs of a family of compatible Poisson structures for the undeformed systems provide a generating function for the integrals in involution of the deformed systems. As a consequence, a family of integrable deformations of the Bogoyavlenskij–Itoh systems was constructed. In \cite{Galajinsky14}, integrable deformations of the Euler top are obtained by alteration of its constants of motion. In the same manner in  \cite{Lazureanu17a,Lazureanu17b,LazPet18}, integrable deformations of some three-dimensional Hamilton-Poisson systems were constructed. The deformed systems are also Hamilton-Poisson systems. 

In \cite{Lazureanu18}, using the fact that a three-dimensional system of differential equations has a Hamilton-Poisson part, a method to construct deformations of such systems was introduced. In particular, this method can be applied to chaotic systems. Note that by this type of deformation the divergence of the initial system remains the same. Therefore, by a deformation of a chaotic system one can obtain a new chaotic system. In \cite{Lazureanu18}, deformations of Lorenz system \cite{Lorenz63}, R\"ossler system \cite{Rossler76}, and Chen system \cite{CheUet99} were given. 

The chaotic systems are widely investigated. The construction of new systems with chaotic behavior is justified by their applications. We recall some applications of such systems, namely in biological systems, secure communication, information processing (see, for example, \cite{BaNiSa85,CheDon98,Chen99,PecCar91,RabAba98,YanChu97}). 

In this paper we consider the $T$ system \cite{Tigan05} given by
\begin{equation}\label{T}
\left\{\begin{array}{l}
\dot x=-ax+ay\\
\dot y=(c-a)x-axz\\
\dot z=-bz+xy
\end{array}\right.~,~a,b,c\in\mathbb{R},a\not=0.
\end{equation}
We obtain deformations of this system. Moreover we study those dynamic properties of a particular deformation that point out its chaotic behavior.

 The $T$ system was analyzed from different points of view: dynamics \cite{Jiang10,TigOpr08,TigCon09,Van}, chaos control \cite{YChen08,Tigan06}, anti-synchronization \cite{VaiRaj11}. 
Therefore further studies of the deformed version of the $T$ system regarding dynamics, complexity, control, and synchronization can be performed (see also \cite{BinLaz15,Chen99,PecCar90,TigLli16}).

The paper is organized as follows. In Section 2 we construct some integrable deformations of the $T$ system. In Section 3 we consider a particular deformation. The new system has a new parameter $g$, which is in fact a tunning parameter. For the values of $a,b,c$ for which the $T$ system is chaotic, we point out the chaotic behavior of the deformed system for a range of values of $g$.

\section{Integrable deformations of the $T$ system}
\smallskip
In this section we construct deformations of the $T$ system that preserves the divergence of its flow. These deformations are obtained using the integrable deformations method for a three-dimensional system of differential equations \cite{Lazureanu18}.

Consider the following dynamical system on an open set $\Omega\subseteq\mathbb{R}^3$ 
\begin{equation}\label{DS}
\dot{{\bf x}}={\bf f}({\bf x})~,~{\bf x}=(x,y,z),{\bf f}=(f_1,f_2,f_3).
\end{equation}
Write $${\bf f}({\bf x})={\bf g}(\bf x)+{\bf h}(\bf x)$$
such that the system
\begin{equation}\label{IS}
\dot{{\bf x}}={\bf g}(\bf x)
\end{equation}
has two functionally independent constants of motion, $H=H({\bf x})$ and $C=C({\bf x})$. Then system \eqref{DS} has the Hamilton-Poisson part \eqref{IS}.

Notice that 
\begin{equation}\label{HPT}
\left\{\begin{array}{l}
\dot x=0\\
\dot y=-axz\\
\dot z=xy
\end{array}\right.~,~a,b,c\in\mathbb{R},a\not=0,
\end{equation}
is a Hamilton-Poisson part of system \eqref{T}. Indeed, we have the following result.
\begin{proposition}\label{p1}
The functions $H,C\in C^\infty(\mathbb{R}^3,\mathbb{R})$ given by
\begin{equation}\label{cm}
H(x,y,z)=x~,~~C(x,y,z)=\frac{1}{2}y^2+\frac{a}{2}z^2,
\end{equation}
are two functionally independent constants of motion of system \eqref{HPT}.
\end{proposition}
\begin{proof}
Using \eqref{HPT}, we immediately obtain $\dfrac{dH}{dt}=0$ and $\dfrac{dC}{dt}=0$, hence $H$ and $C$ are constants of motion. Furthermore, the Jacobian matrix of the functions $H,C$  has the rank two on the set $\{(x,y,z)\in\mathbb{R}^3:y^2+z^2\not=0\}$. Therefore the functions $H$ and $C$ are functionally independent, as required. 
\end{proof}
\begin{remark}
The system \eqref{T} has not an unique Hamilton-Poisson part. Indeed, we can consider ${\bf g}=(ay,-ax-axz,xy)$ and the constants of motion $$H=\frac{1}{2}x^2-az~,~~ C=\frac{1}{2}x^2+\frac{1}{2}y^2+\frac{a}{2}z^2,$$ and also ${\bf g}=(ay,-axz,xy)$ and the constants of motion $$H=\frac{1}{2}x^2-az~,~~C=\frac{1}{2}y^2+\frac{a}{2}z^2.$$
\end{remark}

In order to apply the integrable deformations method for a three-dimensional
system of differential equations \cite{Lazureanu18}, we have
\begin{proposition}\label{p2}
System \eqref{HPT} has the form $\dot{\bf x}=\mu\nabla H\times\nabla C$, where $\mu=x$ and the functions $H$ and $C$ are given by \eqref{cm}.
\end{proposition}
\begin{proof}
System \eqref{HPT} has the form \eqref{IS} $\dot{\bf x}={\bf g}({\bf x})$ with ${\bf g}=(0,-axz,xy)$. On the other hand, $\nabla H\times\nabla C=(0,-az,y)$ and therefore ${\bf g}=x\nabla H\times\nabla C$, which finishes the proof. 
\end{proof}
The next result furnishes the integrable deformations of system $T$.

\begin{thm}\label{t1}
Let $\alpha,\beta$ be arbitrary differentiable functions on $\mathbb{R}^3$. Then an integrable deformation of system (\ref{T}) is given by the following system
\begin{equation}\label{IDT}
\left\{\begin{array}{l}
\dot x=-ax+ay+x(az\alpha_y-y\alpha_z+\alpha_y\beta_z-\alpha_z\beta_y)\\
\dot y=(c-a)x-axz-x(\beta_z+az\alpha_x+\alpha_x\beta_z-\alpha_z\beta_x)\\
\dot z=-bz+xy+x(\beta_y+y\alpha_x+\alpha_x\beta_y-\alpha_y\beta_x)
\end{array}\right.,
\end{equation}
$a,b,c\in\mathbb{R},$ where we denote $f_x:=\dfrac{\partial f}{\partial x}$.
\end{thm} 
\begin{proof}
Let \eqref{IS} be the Hamilton-Poisson part of the three-dimensional system \eqref{DS}. If \eqref{IS} has the form $\dot{\bf x}=\mu\nabla H\times\nabla C$, then an integrable deformations of system \eqref{DS} is given by \cite{Lazureanu18}:
\begin{equation}\label{IDC}
\left\{\begin{array}{lll}
\dot{x}&=&f_1(x,y,z)+\mu\left(H_y\beta_z-H_z\beta_y+\alpha_yC_z-\alpha_zC_y+\alpha_y\beta_z-\alpha_z\beta_y\right)\\
\dot{y}&=&f_2(x,y,z)-\mu\left(H_x\beta_z-H_z\beta_x+\alpha_xC_z-\alpha_zC_x+\alpha_x\beta_z-\alpha_z\beta_x\right)\\
\dot{z}&=&f_3(x,y,z)+\mu\left(H_x\beta_y-H_y\beta_x+\alpha_xC_y-\alpha_yC_x+\alpha_x\beta_y-\alpha_y\beta_x\right)\end{array},
\right.
\end{equation}
In our case $\mu=x$ and the functions $H$ and $C$ are given by \eqref{cm}, whence the conclusion follows.
\end{proof}
\begin{remark}
The functions $\alpha$ and $\beta$ are called deformation functions. It is obvious that if these functions vanish, then \eqref{IDT} becomes the $T$ system.
\end{remark}

\section{The chaotic behavior of a deformed version of the $T$ system}
In this section we point out some dynamic properties of a particular integrable deformation of system (\ref{T}) that emphasize its chaotic behavior. The considered deformation is given by Theorem \ref{t1}, by choosing particular deformation functions $\alpha$ and $\beta$. 

Consider the deformation functions given by $$\alpha (x,y,z)=gz~,~~ \beta(x,y,z)=0,$$ where $g\in\mathbb{R}$ is the deformation parameter. 
We obtain the following particular integrable deformation of the $T$ system:
\begin{equation}\label{particular}
\left\{\begin{array}{l}
\dot x=-ax+ay-gxy\\
\dot y=(c-a)x-axz\\
\dot z=-bz+xy
\end{array}\right.~,~a,b,c\in\mathbb{R},a\not=0,g\in\mathbb{R}.
\end{equation}
We can say that system (\ref{particular}) is the $T$ system with the parametric control $$u(x,y,z)=-gxy,$$ where $g$ is a tunning parameter.
 
In the following we consider $a,b,c\in (0,\infty)$.

It is known that for some values of the parameters $a,b,c$ the $T$ system has chaotic behavior \cite{Tigan05}. Choosing $a=2,b=0.2,c=30$, a strange attractor of this system is shown in Figure \ref{fig:1} (left). 

We mention that the considered deformed version of the $T$ system given by \eqref{particular} also displays a chaotic behavior. Indeed, we notice the presence of a strange attractor for $g>0$ (for example for $g=0.9$; see Figure \ref{fig:1}). Furthermore, for these values, system \eqref{particular} has three hyperbolic equilibrium points: an unstable saddle point $(0,0,0)$ with eigenvalues $\{-8.55,6.55,-0.2\}$, and two unstable saddle-focus points $(-2.42,-1.16, 14)$ and $(1.16,2.42, 14)$ with eigenvalues $\{-2.35,0.6 +3.79 i,0.6 -3.79 i\}$ and $\{-4.43,0.03+1.93 i,0.03-1.93 i\}$, respectively. Using \cite{Sandri96,Sandri10} for the initial conditions $(0.01,0.01,14.01)$, the largest value of positive Lyapunov exponent of the considered system is $\lambda_{L1}=0.17$. 
The other Lyapunov's exponents are given by $\lambda_{L2}=0.0$, $\lambda_{L3}=-3.93$. The Lyapunov dimension or Kaplan-Yorke dimension \cite{KapYor79} of the chaotic attractor of system (\ref{p1}) is 
$$\displaystyle{D_L=2+\frac{\lambda_{L1}+\lambda_{L2}}{|\lambda_{L3}|}=2.044}.$$ Using E$\& $F Chaos program \cite{DiHoPa12} in the Figure \ref{fig:2} the
variation of the largest Lyapunov exponent when $g$ varies is
shown. Apparently, there is a positive Lyapunov exponent for every  $g\in[-1,1]$.
\begin{figure}[t]
\begin{minipage}[t]{0.45\linewidth}
\centering
\includegraphics[width =1\textwidth]{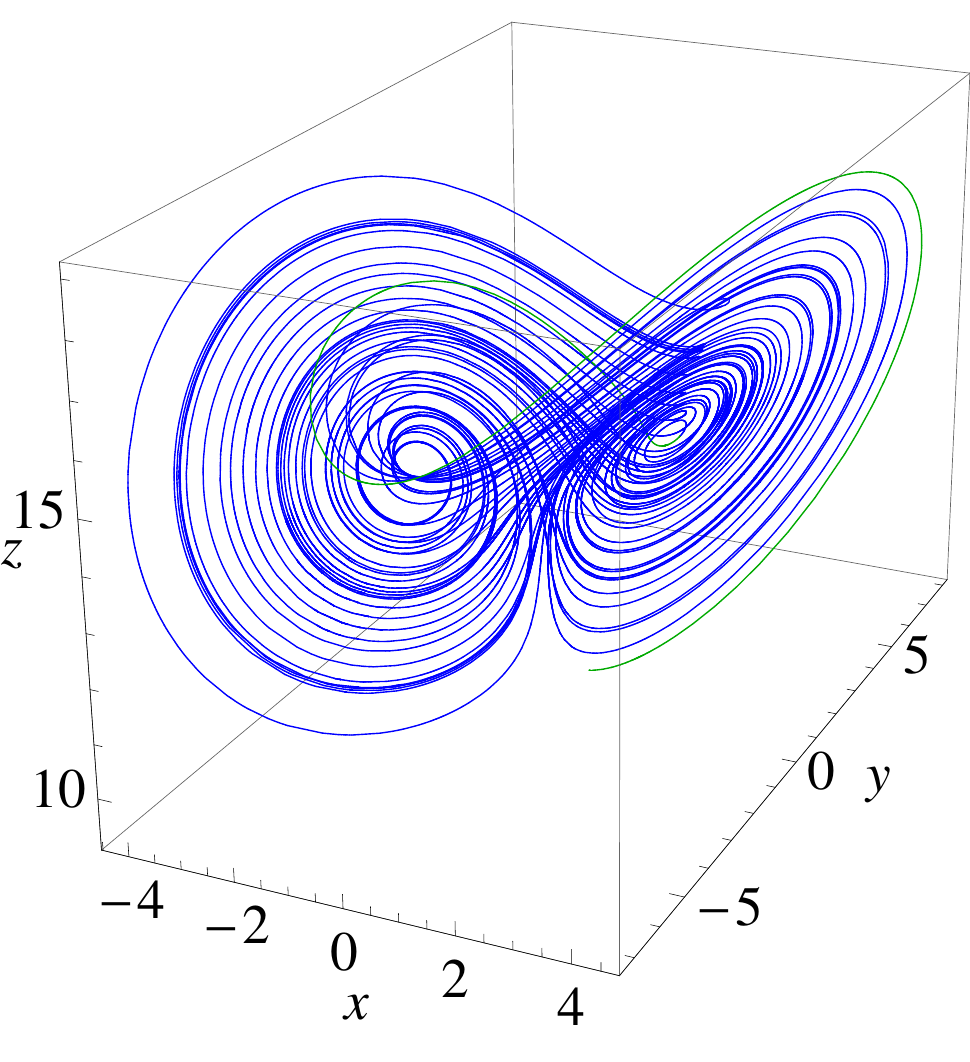}
\end{minipage}\hspace{0.2cm}
\begin{minipage}[t]{0.45\linewidth}\centering
\includegraphics[width =1.1\textwidth]{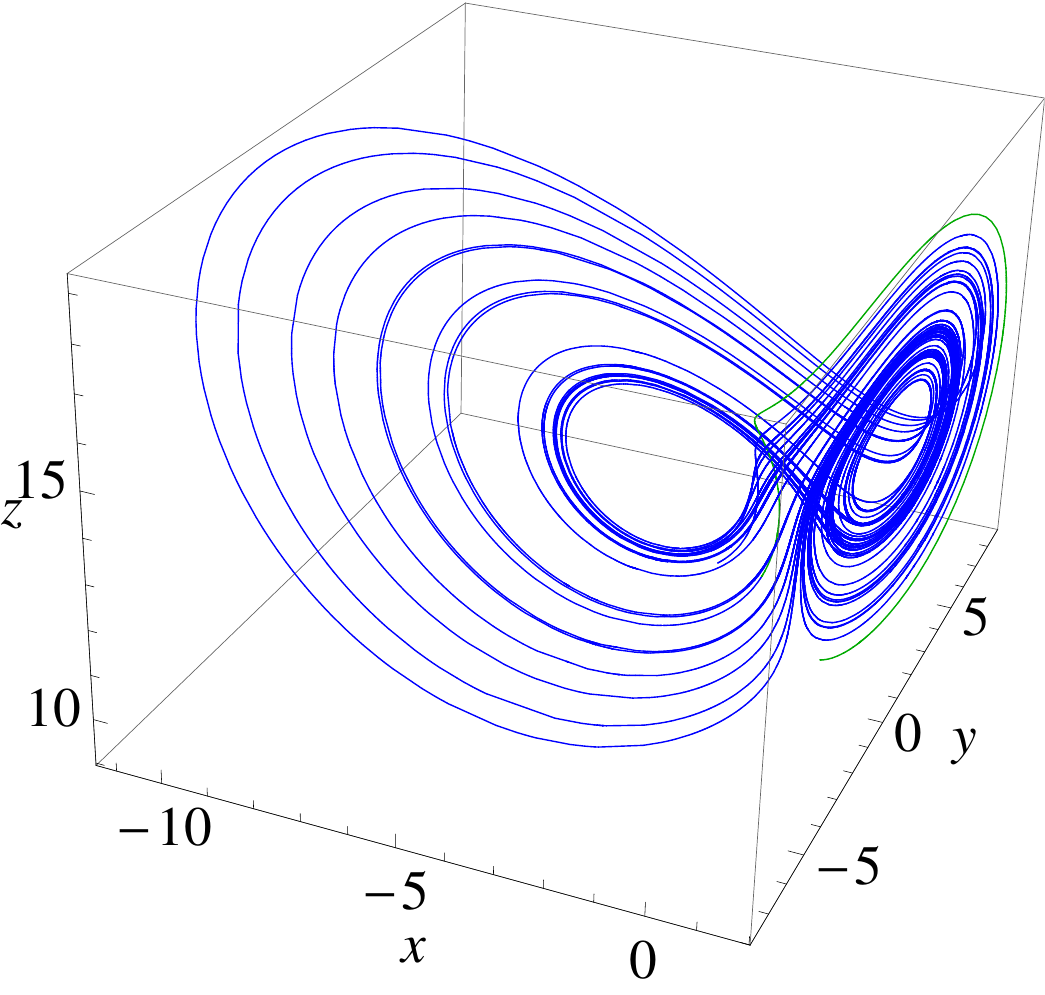}
\end{minipage}
\caption{\footnotesize The initial chaotic attractor (\ref{T}) (left, $g=0$) and one of its deformation (right, $g=0.9$), $a=2,b=0.2,c=30$, initial conditions $(0.01,0.01,14.01)$.}
\label{fig:1}
\end{figure}
\begin{figure}[h!]
\centering
\includegraphics[width =0.8\textwidth]{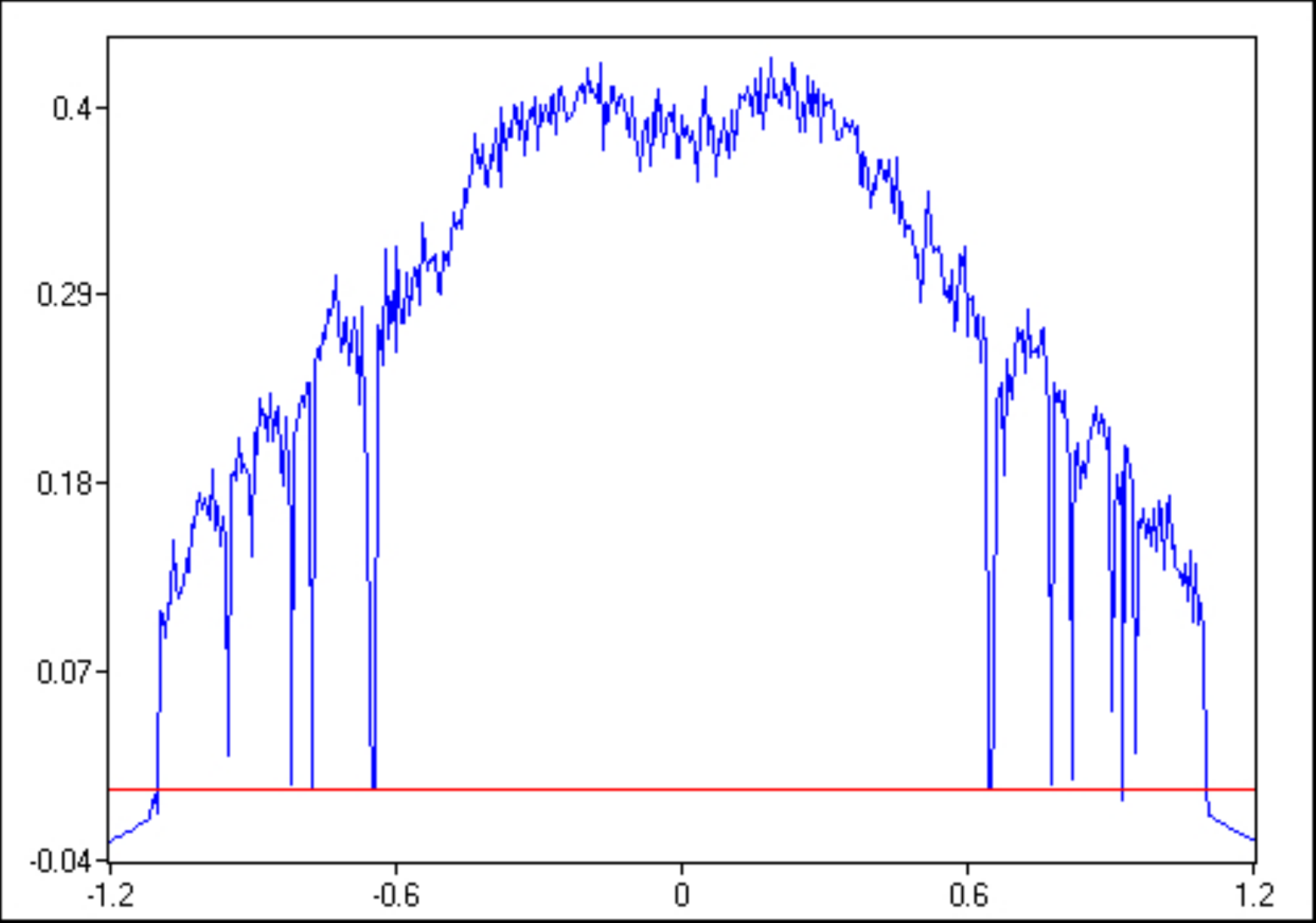}
\caption{\footnotesize The variation of the largest Lyapunov exponent of system \eqref{particular} as function of $g\in[-1.2,1.2]$ ( $a=2,b=0.2,c=30$, initial conditions $(0.01,0.01,14.01)$).}
\label{fig:2}
\end{figure}

\begin{remark}
By geometric point of view the orbits plotted in Figure \ref{fig:1} show that system \eqref{particular} is a deformation of system $T$.
\end{remark}

\section{ACKNOWLEDGMENTS}
This work was supported by research grants PCD-TC-2017.

\smallskip
{\footnotesize
\hspace*{0.5cm}
\begin{minipage}[t]{8cm}$$\begin{array}{l}
\mbox{Cristian L\u azureanu -- Department of Mathematics,}\\
\mbox{Politehnica University of Timi\c soara},\\
 \mbox{P-ta Victoriei 2, 300 006, Timi\c soara, ROMANIA}\\
\mbox{E-mail: cristian.lazureanu@upt.ro}\end{array}$$
\end{minipage}\\
{\footnotesize
\hspace*{0.5cm}
\begin{minipage}[t]{8cm}$$\begin{array}{l}
\mbox{Cristiana C\u aplescu -- Department of Mathematics,}\\
\mbox{Politehnica University of Timi\c soara},\\
 \mbox{P-ta Victoriei 2, 300 006, Timi\c soara, ROMANIA}\\
\mbox{E-mail: cristiana.caplescu@upt.ro}\end{array}$$
\end{minipage}\\

\end{document}